\documentclass[12pt,a4paper]{article}

\usepackage{latexsym}
\usepackage{amsmath}
\usepackage{amssymb}
\usepackage{amsthm}
\usepackage{amscd}
\usepackage{mathrsfs}
\usepackage[all]{xy}
\usepackage{graphicx}
\usepackage{comment}
\usepackage{arydshln}

\newtheorem{thm}{Theorem}[section]
\newtheorem{prop}[thm]{Proposition}
\newtheorem{defn}[thm]{Definition}
\newtheorem{lem}[thm]{Lemma}
\newtheorem{cor}[thm]{Corollary}

\newtheorem{rem}[thm]{Remark}

\newtheorem*{thmn}{Theorem}

\theoremstyle{remark}

\makeatletter

\@addtoreset{equation}{section}
\makeatother

\makeatletter

\@addtoreset{equation}{section}
\makeatother

\makeatletter
\newcommand{\subsubsubsection}{\@startsection{paragraph}{4}{\z@}%
 {1.0\Cvs \@plus.5\Cdp \@minus.2\Cdp}%
 {.1\Cvs \@plus.3\Cdp}%
 {\reset@font\sffamily\normalsize}
 }
\makeatother
\DeclareMathOperator{\Fr}{Fr}
\DeclareMathOperator{\Gal}{Gal}
\DeclareMathOperator{\id}{id}
\DeclareMathOperator{\tr}{tr}

\DeclareMathOperator{\Ker}{Ker}

\DeclareMathOperator{\Aut}{Aut}

\DeclareMathOperator{\Tr}{Tr}
\DeclareMathOperator{\Nr}{Nr}

\DeclareMathOperator{\Image}{Im}

\DeclareMathOperator{\GL}{GL}

\DeclareMathOperator{\HU}{HU}

\DeclareMathOperator{\oU}{U}

\DeclareMathOperator{\oH}{H}

\newcommand{\bF}{\mathbb{F}}

\newcommand{\bQ}{\mathbb{Q}}

\newcommand{\bfH}{\mathbf{H}}

\newcommand{\bfU}{\mathbf{U}}

\newcommand{\cf}{\textit{cf.\ }}

\usepackage[usenames,dvipsnames]{color}

\begin{document}


\title
{Shintani lifts for Weil representations\\ 
of unitary groups over finite fields} 

\author{Naoki Imai and Takahiro Tsushima}
\date{}
\maketitle
\begin{abstract}
We construct extended Weil representations of unitary groups 
over finite fields geometrically, 
and show that they are Shintani lifts for Weil representations. 
\end{abstract}

\section{Introduction}\label{intro}
Let $q$ be a power of a prime number $p$. 
Let $n$ be a positive integer. 
We define 
\[
\oU_n(q)=\{g \in \mathrm{GL}_n(\mathbb{F}_{q^2})
\mid g^\ast g=\mathrm{I}_n\}, 
\]
where $g^\ast=(a_{j,i}^q)$ for $g=(a_{i,j})$. 
Let $\ell \neq p$ be a prime number. 
Let $\psi$ be a non-trivial character of 
$\mathbb{F}_{q,+}=\{x \in \mathbb{F}_{q^2} \mid x^q+x=0\}$ 
over $\overline{\mathbb{Q}}_{\ell}$. 
A Weil representation of 
a unitary group $\oU_n(q)$ associated to $\psi$ is constructed in \cite{GerWeil}, which we denote by 
$\rho_{\oU_n(q),\psi}$. 
Let $m$ be a positive odd integer. 

In this paper, we investigate 
the behavior of $\rho_{\oU_n(q),\psi}$ via Shintani 
lifting from $\mathbb{F}_q$ to $\mathbb{F}_{q^m}$. 
Let $\Gamma$ be a cyclic group of order $m$ with 
generator $\sigma$. 
Let $F \colon \oU_n(q^m) \to \oU_n(q^m);\ 
g \mapsto ({}^t g^{\tau})^{-1}$, where 
$\tau((a_{i,j}))=(a_{i,j}^q)$.
Let $\sigma$ act on $\oU_n(q^m)$ as $F$. 
We consider the semidirect group 
$\oU_n(q^m) \rtimes \Gamma$. 
In \cite{GyoLift}, Gyoja defines a norm map 
from 
the set of $\oU_n(q^m)$-conjugacy classes in 
$\oU_n(q^m) \sigma$ to 
the set of conjugacy classes in 
$\oU_n(q)$, which is denoted by $\mathrm{N}$.
We set $\psi_m=\psi \circ \Tr_{\mathbb{F}_{q^{2m}}/\mathbb{F}_{q^2}} \colon \mathbb{F}_{q^m,+} \to 
\overline{\mathbb{Q}}_{\ell}^{\times}$.
\begin{thmn}
There is a unique extension 
$\widetilde{\rho}_{\oU_n(q^m),\psi_m}$ of 
$\rho_{\oU_n(q^m),\psi_m}$ to 
$\oU_n(q^m) \rtimes \Gamma$ such that 
\begin{equation*}\label{int}
\tr \widetilde{\rho}_{\oU_n(q^m),\psi_m}(g,\sigma)
=\tr \rho_{\oU_n(q),\psi}(\mathrm{N}(g,\sigma))
\end{equation*}
for any $g \in \oU_n(q^m)$. 
\end{thmn}
This is a version of \cite[Theorem in 
\S1]{HeWaWeil} in unitary cases. 
Similarly to \cite[Theorem 4.1]{HeWaWeil}, we actually show 
a stronger equality. 
Namely, we show a similar equality for Heisenberg--Weil representations 
and more general norm maps. 
See Theorem \ref{main} for the precise statement. 
As mentioned in \cite[\S1]{HeWaWeil}, 
the arguments in \cite{HeWaWeil} work also in unitary cases. 
In this paper, we use an inductive argument similar 
to that in \cite{HeWaWeil}, 
but replace some ingredients in the proof with geometric inputs: 
In \cite{HeWaWeil}, an extended Weil representation 
is constructed based on the 
Schr\"{o}dinger model of a Weil representation. 
On the other hand, 
we construct $\widetilde{\rho}_{\oU_n(q^m),\psi_m}$ 
in a geometric way and study the representation by geometric tools. 

In \cite[Theorem 2.5]{ITGeomHW}, it is known that 
the representation ${\rho}_{\oU_n(q^m),\psi_m}$ is realized in the $\psi_m$-isotypic part of the middle $\ell$-adic cohomology of 
the smooth affine variety over $\mathbb{F}_{q^2}$
defined by 
\begin{equation*}\label{ft}
 z^{q^m}+z=\sum_{i=1}^n x_i^{q^m+1} 
\end{equation*}
in $\mathbb{A}_{\mathbb{F}_{q^2}}^{n+1}$. 
In order to extend this representation to 
$\oU_n(q^m) \rtimes \Gamma$, 
we use the Frobenius action over $\mathbb{F}_{q^2}$. 
A merit in our geometric construction is that the the action of the cyclic group $\Gamma$ in the Shintani lift appears very naturally from the rationality of the above variety. 
We believe that this gives a new insight on the relation between liftings of representations and geometry.

We briefly explain the content of each section. 
In \S \ref{sec:norm}, we collect several known 
facts on Gyoja's norm map. 
In \S \ref{Pre}, we recall unitary Heisenberg--Weil representations. 
In \S \ref{GC}, we construct an extension 
$\widetilde{\rho}_{\HU_n(q^m),\psi_m}$ of 
a Heisenberg--Weil representation 
geometrically and show Theorem \ref{main} when $n=1$ 
in Proposition \ref{prop:trzFr}. 
In \S \ref{od}, we study 
the behavior of $\widetilde{\rho}_{\HU_n(q^m),\psi_m}$ restricted to 
several subgroups. 
In \S \ref{Sup}, we study 
the support of the trace of $\widetilde{\rho}_{\HU_n(q^m),\psi_m}$. 
In \S \ref{sec:Main}, we state our main theorem. 
Under the knowledges in \S \ref{od}--\S \ref{Sup}, we reduce 
Theorem \ref{main} to a special case in Lemma \ref{inv} in \S \ref{Red}. 
By using the character formula of a tensor induction 
in \cite{GlIsTen}, 
we show Theorem \ref{main} in the special case 
in \S \ref{Prred}.

\subsection*{Notation}
\begin{itemize}
\item 
Let $q$ be a power of a prime number $p$. 
\item Let $\overline{\bF}_q$ be an algebraic closure of $\mathbb{F}_q$. 
For a scheme $X$ over $\mathbb{F}_q$ and an endomorphism 
$F$ of $X_{\overline{\bF}_q}$, let 
\[
 X^F=\{x \in X(\overline{\bF}_q) \mid F(x)=x\}.
\]
\item 
For a positive integer $m$, let $\Fr_{q^m}$ denote the geometric Frobenius element of $\Gal (\overline{\bF}_q/\bF_{q^m})$ given by $\overline{\bF}_q \to \overline{\bF}_q;\ x \mapsto x^{q^{-m}}$. 
\item 
Let $\ell$ be a prime number different from $p$. 
\item 
For an abelian group $A$, let 
$A^{\vee}=\mathrm{Hom}_{\mathbb{Z}}
 (A,\overline{\mathbb{Q}}_{\ell}^{\times})$. 
\item 
For a representation $W$ over $\overline{\mathbb{Q}}_{\ell}$ of an abelian group $A$ 
and $\chi \in A^{\vee}$, 
let $W[\chi]$ denote the $\chi$-isotypic part of $W$. 
\end{itemize}

\section{Norm map}\label{sec:norm}
We follow \cite[\S3]{HeWaWeil} and \cite[\S3]{GyoLift}. 
Let $\mathbf{G}$ be a connected algebraic group 
 over $\mathbb{F}_q$ with Frobenius 
endomorphism $F$. 
Let $m$ be a positive integer. 
We put $G=\mathbf{G}^{F^m}$. 

Let $\Gamma$ be a cyclic group of order $m$ with a generator 
$\sigma$.
Let $\sigma$ act on $G$ as 
$F$. 
We consider the semidirect group $G \rtimes \Gamma$. 
Let $1 \leq i \leq m$ be an integer. 
We set $d=(m,i)$. We choose an integer $t$ such that
$ti \equiv d \pmod m$.  We set $\mu=m/d$. 
For $g \in G$, we can choose 
\[
\alpha=\alpha(g) \in \mathbf{G}(\overline{\bF}_q) 
\]
such that 
\[
\alpha^{-1} F^d(\alpha)=(g,\sigma^i)^t (1,\sigma^{-it}) 
\]
by Lang's theorem. 
We put 
\[
\mathrm{N}_{i,t}(g,\sigma^i)=\alpha g F^i(g) \cdots F^{i(\mu-1)}(g) 
\alpha^{-1}. 
\]
The element $\mathrm{N}_{i,t}(g,\sigma^i)$
does belong to $G^{F^i}$. Its conjugacy class 
does not depend on the choice of $\alpha$. 
By \cite[Lemma 3.2(1)]{GyoLift}, $\mathrm{N}_{i,t}$ 
induces a bijection from the set of $G$-conjugacy classes in $G  \sigma^i$ to the set of 
conjugacy classes in $G^{F^i}$.  
The norm map is originally defined in \cite{ShiTworem}
in the case where $i=t=1$ 
and generalized in \cite{KawIrruni}. 

Let $C$ be an algebraically closed field of characteristic $0$. 
Let $\mathcal{C}(G  \sigma^i)$ denote the 
vector space of $C$-valued functions on 
$G  \sigma^i$ which are invariant under conjugation by $G$. 
For a finite group $H$, let $\mathcal{C}(H)$ denote the 
set of $C$-valued class functions on $H$. 
Then we have the bijection 
\[
\mathcal{N}_{i,t} \colon \mathcal{C}(G^{F^i}) \to \mathcal{C}(G  \sigma^i);\ 
\chi \mapsto \chi \circ \mathrm{N}_{i,t}. 
\]
This induces the bijection 
\[
 \mathcal{C}(G^{F^i})^F \stackrel{\sim}{\longrightarrow} \mathcal{C}(G  \sigma^i)^{\sigma} 
\]
by \cite[Lemma 3.2(2)]{GyoLift}, 
where $\sigma$ acts on $G  \sigma^i$ by the conjugation. 
Note that $G^{F^i}=\mathbf{G}^{F^d}$. 
\begin{lem}\label{gy}
Let $\mathbf{H}$ be a connected algebraic subgroup of 
$\mathbf{G}$ defined over $\bF_q$, and $\mathbf{G}_1,\mathbf{G}_2$
two connected algebraic groups over $\mathbb{F}_q$ with 
Frobenius endomorphisms $F$. Let 
$H=\mathbf{H}^{F^m}$ and 
$G_j=\mathbf{G}_j^{F^m}$ for $j =1,2$.
\begin{enumerate}
\item\label{enu:HInd} 
Let $\widetilde{\chi} \in \mathcal{C}(H \rtimes \Gamma)$, 
and take $\chi \in \mathcal{C}(H^{F^i})^F$ 
such that $\widetilde{\chi}|_{H  \sigma^i}=\mathcal{N}_{i,t}(\chi)$. 
Then we have 
\[
(\mathrm{Ind}_{H \rtimes \Gamma}^{G \rtimes \Gamma}\widetilde{\chi})|_{G  \sigma^i}
=\mathcal{N}_{i,t}(\mathrm{Ind}_{H^{F^i}}^{G^{F^i} }\chi). 
\]
\item 
Let 
$\widetilde{\chi} \in \mathcal{C}(G \rtimes \Gamma)$, 
and take $\chi \in \mathcal{C}(G^{F^i})^F$ 
such that $\widetilde{\chi}|_{G \sigma^i}
=\mathcal{N}_{i,t}(\chi)$. 
Then we have 
\[
(\widetilde{\chi}|_{H \rtimes \Gamma})|_{H \sigma^i}=\mathcal{N}_{i,t}(\chi|_{H^{F^i}}). 
\]
\item 
Assume that $\mathbf{H}$ is a normal algebraic subgroup of $\mathbf{G}$. 
We identify $G/H$ with $(\mathbf{G}/\mathbf{H})^{F^m}$. 
Let $\widetilde{\chi} \in \mathcal{C}((G/H) \rtimes \Gamma)$, 
and take $\chi \in \mathcal{C}((G/H)^{F^i})^F$ 
such that 
$\widetilde{\chi}|_{(G/H)  \sigma^i}
=\mathcal{N}_{i,t}(\chi)$. 
Let 
\[
p' \colon G \rtimes \Gamma \to (G/H) \rtimes \Gamma, 
\quad 
p \colon G^{F^i} \to (G/H)^{F^i}
\]
be the canonical projections. 
Then we have 
\[
(\widetilde{\chi} \circ p')|_{G  \sigma^i}=\mathcal{N}_{i,t}(\chi \circ p). 
\]
\item\label{enu:chiprod} 
Let $\widetilde{\chi}_j \in \mathcal{C}(G_j \rtimes \Gamma)$ 
for $j=1,2$, and take 
$\chi_j \in \mathcal{C}(G_j^{F^i})^F$ 
such that 
$\widetilde{\chi}_j|_{G_j  \sigma^i}
=\mathcal{N}_{i,t}(\chi_j)$. 
Then we have 
\[
(\widetilde{\chi}_1 \times \widetilde{\chi}_2)|_{
(G_1 \times G_2)  \sigma^i}
=\mathcal{N}_{i,t}(\chi_1 \times \chi_2), 
\]
where 
$(G_1 \times G_2) \sigma^i$ is 
identified with 
\[ 
\bigl\{ \bigl( (g_1,\sigma^i),(g_2,\sigma^i) \bigr) \in (G_1 \rtimes \Gamma) \times (G_2 \rtimes  \Gamma) \bigm| g_1 \in G_1 , \ g_2 \in G_2 \bigr\}. 
\]
\end{enumerate}
\end{lem}
\begin{proof}
This is proved in \cite[Lemma 3.6]{GyoLift} if $C=\mathbb{C}$. 
The same proof works also for a general $C$. 
\end{proof}

\section{Heisenberg--Weil representation}\label{Pre}
Let $V$ be a vector space over $\mathbb{F}_{q^2}$ 
with a hermitian form $h$. 
In this paper, a hermitian form is supposed to be sesquilinear on the first coordinate. 
For the hermitian space $(V,h)$, 
we define $\bfH_{(V,h)}$ over $\bF_q$ as 
\[
 \bfH_{(V,h)} (R) = \{ (v,a) \in 
 (V \otimes_{\bF_q} R ) \times (\bF_{q^2} \otimes_{\bF_q} R) 
 \mid (\Fr_q \otimes \id_R) (a) +a =h_R (v,v) \} 
\]
for an $\bF_q$-algebra $R$, 
where 
$h_R$ is the $R$-linear extension of $h$, 
with the group operation
\[
 (v,a)(v',a')=(v+v', a+a'+h_R(v,v')). 
\] 
We put 
$\oH (V,h) =\bfH_{(V,h)} (\bF_q)$ 
and call it 
the unitary Heisenberg group associated to $(V,h)$. 

Assume that $h$ is nondegenerate in the sequel. 
Let $\bfU_{(V,h)}$ be the unitary algebraic group over $\bF_q$ 
defined as 
\begin{align*}
 \bfU_{(V,h)} &(R)\\  
 = \{ & g \in \Aut_{\bF_{q^2} \otimes_{\bF_q} R} 
 (V \otimes_{\bF_q} R)
 \mid h_R (gv_1,gv_2) =h_R (v_1,v_2) \ 
 \textrm{for $v_1, v_2 \in V \otimes_{\bF_q} R$} \} 
\end{align*}
for an $\bF_q$-algebra $R$, 
with an action on $\bfH_{(V,h)}$ defined by 
$(v,a) \mapsto (gv,a)$ for 
$(v,a) \in \bfH_{(V,h)}(R)$ and 
$g \in \bfU_{(V,h)}(R)$. 
We put 
\[
\mathbf{HU}_{(V,h)}=\bfH_{(V,h)} \rtimes \bfU_{(V,h)}. 
\]
Further we put 
\[
 \oU (V,h) =\bfU_{(V,h)}(\bF_q), \quad 
 \HU(V,h)=\mathbf{HU}_{(V,h)}(\bF_q). 
\] 
We simply write 
$\bfH_V$, 
$\bfU_V$ and 
$\mathbf{HU}_V$ 
for 
$\bfH_{(V,h)}$, 
$\bfU_{(V,h)}$ and 
$\mathbf{HU}_{(V,h)}$ 
if it is clear which hermitian form is involved.

For a positive integer $m$, we put 
\[
 \mathbb{F}_{q^m,+}=\{x \in \overline{\mathbb{F}}_q \mid x^{q^m}+x=0\}. 
\] 
Let $\psi \in \mathbb{F}_{q,+}^{\vee} \setminus \{1\}$. 
Let $\rho_{\oH (V,h),\psi}$ denote the unique 
irreducible representation of $\oH (V,h)$ 
whose restriction to the center $Z(V,h)$ contains $\psi$. 
We put $n=\dim_{\mathbb{F}_{q^2}} V$. Then 
$\rho_{\oH (V,h),\psi}|_{Z(V,h)} \simeq \psi^{\oplus q^n}$. 
The following is shown in \cite{GerWeil} 
and stated in \cite[Lemma 2.2]{ITGeomHW}. 

\begin{lem}\label{lem:hw}
There exists a unique representation 
$\rho_{\HU(V,h),\psi}$ of 
$\HU(V,h)$ characterized by 
\begin{itemize}
\item 
an isomorphism 
$\rho_{\HU(V,h),\psi}|_{\oH (V,h)} \simeq 
\rho_{\oH (V,h),\psi}$, and 
\item 
the equality 
$\tr \rho_{\HU(V,h),\psi}(g)=(-1)^n 
 (-q)^{N(g)}$ for $g \in \oU (V,h)$, where 
\[
N(g)=\dim_{\mathbb{F}_{q^2}} \ker (g-\id_V ). 
\]
\end{itemize}
\end{lem}
We call $\rho_{\HU(V,h),\psi}$ 
the Heisenberg--Weil representation of $\HU(V,h)$
associated to $\psi$. 
The restriction of $\rho_{\HU(V,h),\psi}$
to the subgroup $\oU (V,h)$ is called 
the Weil representation of $\oU (V,h)$ 
associated to $\psi$. 

Let 
\[
 h_n \colon 
 \mathbb{F}_{q^2}^n \times \mathbb{F}_{q^2}^n \to \mathbb{F}_{q^2};\ 
 ((x_1,\ldots,x_n),(x'_1,\ldots,x'_n)) \mapsto 
 \sum_{i=1}^n x_i^q x'_i.
\]
Any nondegenerate hermitian space $(V,h)$ of dimension $n$ is isometric to $(\mathbb{F}_{q^2}^n,h_n)$. 
We simply write $\oH_n(q)$, $\oU_n(q)$ and 
$\HU_n(q)$ for 
$\oH (\mathbb{F}_{q^2}^n,h_n)$, $\oU (\mathbb{F}_{q^2}^n,h_n)$ and 
$\HU(\mathbb{F}_{q^2}^n,h_n)$, respectively.

\section{Geometric construction}\label{GC}
We give a 
geometric construction of an extended Heisenberg--Weil representation. 
Let $m$ be an odd positive integer. 
We consider the smooth affine variety over $\mathbb{F}_{q^2}$
defined by 
\begin{equation*}\label{val}
z^{q^m}+z=\sum_{i=1}^n x_i^{q^m+1}
\end{equation*}
in $\mathbb{A}_{\mathbb{F}_{q^2}}^{n+1}$, which is denoted by $X_{m,n}$.
The group $\HU_n(q^m)$ acts on $X_{m,n,\mathbb{F}_{q^{2m}}}$ by 
\begin{align*}
 & ((x_i),z) \mapsto (g(x_i),z) \quad 
 \textrm{for $g \in \oU_n(q^m)$}, \\
 & ((x_i),z) \mapsto 
 ((x_i)+v,z+a+h_n(v,(x_i))) \quad 
 \textrm{for $(v,a) \in \oH_n(q^m)$}. 
\end{align*}
Let $\HU_n(q^m)$ act on $H_{\rm c}^n(X_{m,n,\overline{\mathbb{F}}_q},\overline{\mathbb{Q}}_{\ell})$ by letting $g \in \HU_n(q^m)$ act as $(g^{-1})^\ast$.

Let $\nu \colon 
\mathrm{Gal}(\overline{\mathbb{F}}_q/\mathbb{F}_{q^2}) \to 
\overline{\mathbb{Q}}_{\ell}^{\times}$ be the character 
which sends the geometric Frobenius element $\mathrm{Fr}_{q^2}$ to $-q^{-1}$. 
Let $\psi \in \mathbb{F}_{q,+}^{\vee} \setminus \{1\}$. 
The restriction of the trace map 
$\Tr_{\mathbb{F}_{q^{2m}}/\mathbb{F}_{q^2}}$ to 
$\mathbb{F}_{q^m,+}$ induces 
$\Tr_{\mathbb{F}_{q^{2m}}/\mathbb{F}_{q^2}} \colon 
\mathbb{F}_{q^m,+} \to \mathbb{F}_{q,+}$. 
We set 
\[
\psi_m =\psi \circ \Tr_{\mathbb{F}_{q^{2m}}/\mathbb{F}_{q^2}} \in \mathbb{F}_{q^m,+}^{\vee}. 
\]

The Frobenius endomorphism $F$ of the algebraic group $\mathbf{HU}_{(\mathbb{F}_{q^2}^n,h_n)}$ over $\bF_q$ induces an automorphism of 
$\HU_n(q^m)$, for which we use the same symbol $F$. Viewing $\HU_n(q^m)$ as a subset of $\mathbb{F}_{q^{2m}}^n \times \mathbb{F}_{q^{2m}} \times \GL_n (\mathbb{F}_{q^{2m}})$ naturally, the coordinatewise $q$-th power map on $\mathbb{F}_{q^{2m}}^n \times \mathbb{F}_{q^{2m}} \times \GL_n (\mathbb{F}_{q^{2m}})$ induces an automorphism $\tau$ of $\HU_n(q^m)$. Then we have $F^2=\tau^2$ by the definitions. Hence we have $F=F^{m+1}=\tau^{m+1}$ as an automorphism of $\HU_n(q^m)$ because $\HU_n(q^m)=\mathbf{HU}_{(\mathbb{F}_{q^2}^n,h_n)}^{F^m}$ and $m$ is odd. 

We regard $H_{\rm c}^n(X_{m,n,\overline{\mathbb{F}}_q},\overline{\mathbb{Q}}_{\ell})$
as a $\mathrm{Gal}(\overline{\mathbb{F}}_q/\mathbb{F}_{q^2})$-representation as usual. 
By \cite[Rapport (1.8.1)]{DelCoet}, the action of $\Fr_{q^2}$ on $H_{\rm c}^n(X_{m,n,\overline{\mathbb{F}}_q},\overline{\mathbb{Q}}_{\ell})$ is equal to the pullback under the $q^2$-nd power Frobenius endomorphism $F_X$ of $X_{m,n}$ over $\bF_{q^2}$. 
Let $\mathrm{Gal}(\overline{\mathbb{F}}_q/\mathbb{F}_{q^2})$ 
act on $\HU_n(q^m)$ by letting 
$\mathrm{Fr}_{q^2}$ act as $F^{-2}$, 
since we have 
\[
 F_X^* (g^{-1})^*=(g^{-1}F_X)^*=(F_X F^{-2}(g)^{-1})^*=(F^{-2}(g)^{-1})^* F_X^* . 
\]
We regard 
\[
H_{\rm c}^n(X_{m,n,\overline{\mathbb{F}}_q},\overline{\mathbb{Q}}_{\ell})[\psi_m] \otimes \nu^n 
\] as an  
$\HU_n(q^m) \rtimes \mathrm{Gal}(\overline{\mathbb{F}}_q/\mathbb{F}_{q^2})$-representation, which we denote by $\widetilde{\rho}_{\HU_n(q^m),\psi_m}$. 

\begin{lem}[{\cite[Theorem 2.5]{ITGeomHW}}]\label{lem:HWg}
We have 
$H_{\rm c}^i(X_{m,n,\overline{\mathbb{F}}_q},\overline{\mathbb{Q}}_{\ell})
 [\psi_m]=0$ for $i \neq n$. 
The restriction of $\widetilde{\rho}_{\HU_n(q^m),\psi_m}$ 
to $\HU_n(q^m)$ 
is isomorphic to $\rho_{\HU_n(q^m),\psi_m}$. 
\end{lem}

\begin{lem}\label{lem:FrS}
\begin{enumerate}
\item\label{enu:X11} 
The geometric Frobenius element 
$\mathrm{Fr}_{q^2}$ acts on 
\[
H_{\rm c}^1(X_{1,1,\overline{\mathbb{F}}_q},\overline{\mathbb{Q}}_{\ell})[\psi]
\] 
as scalar multiplication by $-q$.  
\item\label{enu:Fr2m} 
We have 
$\widetilde{\rho}_{\HU_n(q^m),\psi_m}(\mathrm{Fr}_{q^{2m}})=1$. 
\end{enumerate}
\end{lem}
\begin{proof}
By Lemma \ref{lem:HWg}, we have 
$H_{\rm c}^i(X_{1,1,\overline{\mathbb{F}}_q},\overline{\mathbb{Q}}_{\ell})[\psi]=0$ for 
$i \neq 1$. 
Hence we have 
 \[
 \Tr(\mathrm{Fr}_{q^2};H_{\rm c}^1
 (X_{1,1,\overline{\mathbb{F}}_q},\overline{\mathbb{Q}}_{\ell})[\psi])
 =-q^{-1}\sum_{\eta \in \mathbb{F}_{q,+}} \psi(\eta)
 \Tr((-\eta)\mathrm{Fr}_{q^2};H_{\rm c}^{\ast}
 (X_{1,1,\overline{\mathbb{F}}_q},\overline{\mathbb{Q}}_{\ell})). 
 \]
By the Lefschetz fixed point formula, we have 
\begin{align*}
 \Tr((-\eta)\mathrm{Fr}_{q^2};H_{\rm c}^{\ast}
 (X_{1,1,\overline{\mathbb{F}}_q},\overline{\mathbb{Q}}_{\ell})) &= 
 \sharp \{ (x,z) \in X_{1,1}(\overline{\mathbb{F}}_q) \mid 
 x^{q^2}=x,\ z^{q^2} -\eta =z \} \\ 
 &=
 \begin{cases}
 q^3 & \textrm{if $\eta=0$}, \\
 0 & \textrm{otherwise}. 
 \end{cases}
\end{align*}
Thus we obtain 
\begin{equation}\label{thus}
 \Tr(\mathrm{Fr}_{q^2};H_{\rm c}^1
 (X_{1,1,\overline{\mathbb{F}}_q},\overline{\mathbb{Q}}_{\ell})[\psi])=-q^2. 
\end{equation}
The action of $\mathrm{Fr}_{q^2}$ commutes with the one of 
$\oH_1(q)$. Since $H_{\rm c}^1(X_{1,1,\overline{\mathbb{F}}_q},\overline{\mathbb{Q}}_{\ell})[\psi]$
 is an irreducible $\oH_1(q)$-representation 
 by Lemma \ref{lem:HWg}, 
$\mathrm{Fr}_{q^2}$ acts on it as scalar multiplication by Schur's lemma. 
Hence the first claim follows from \eqref{thus}, because $\dim H_{\rm c}^1(X_{1,1,\overline{\mathbb{F}}_q},\overline{\mathbb{Q}}_{\ell})[\psi]=q$ by Lemma \ref{lem:HWg}.

We show the second claim. 
By the first assertion and the K\"{u}nneth formula, 
$\mathrm{Fr}_{q^{2m}}$ acts on 
$H_{\rm c}^n(X_{m,n,\overline{\mathbb{F}}_q},\overline{\mathbb{Q}}_{\ell})[\psi]$ as scalar multiplication by $(-1)^nq^{mn}$. 
We have 
$\nu^n(\mathrm{Fr}_{q^{2m}})=(-1)^{mn} q^{-mn}$. 
Since $m$ is odd, the claim follows.   
\end{proof}

We set $\Gamma=\mathrm{Gal}(\mathbb{F}_{q^{2m}}/\mathbb{F}_{q^2})$. 
By Lemma \ref{lem:FrS} \ref{enu:Fr2m}, we can regard $\widetilde{\rho}_{\HU_n(q^m),\psi_m}$
as an $\HU_n(q^m) \rtimes \Gamma$-representation. 
This is an extension of 
$\rho_{\HU_n(q^m),\psi_m}$ to $\HU_n(q^m) \rtimes \Gamma$ by Lemma \ref{lem:HWg}. 
The restriction of $\widetilde{\rho}_{\HU_n(q^m),\psi_m}$
to the subgroup $\oU_n(q^m) \rtimes \Gamma$
is denoted by $\widetilde{\rho}_{\oU_n(q^m),\psi_m}$ as in \S\ref{intro}. 
We put 
\[
\sigma=\mathrm{Fr}_{q^2}^{\frac{m-1}{2}}. 
\]
Then $\sigma$ is the generator of $\Gamma$ 
acting on $\HU_n(q^m)$ as $F$.

\begin{defn}\label{def:Vh}
Let $(V,h)$ be a nondegenerate 
hermitian space of dimension $n$ over $\mathbb{F}_{q^2}$. 
We take an isometry 
$(V,h) \simeq (\mathbb{F}_{q^2}^n,h_n)$. 
This induces $\mathbf{HU}_{(V,h)}(\bF_{q^m}) \rtimes \Gamma 
\simeq \HU_n(q^m) \rtimes \Gamma$.  
Via this isomorphism and $\widetilde{\rho}_{\HU_n(q^m),\psi_m}$, we define 
a representation of $\mathbf{HU}_{(V,h)}(\bF_{q^m}) \rtimes \Gamma$, which 
is denoted by $\widetilde{\rho}_{\mathbf{HU}_{(V,h)}(\bF_{q^m}),\psi_m}$. 
\end{defn}

\begin{rem}
The isomorphism class of 
$\widetilde{\rho}_{\mathbf{HU}_{(V,h)}(\bF_{q^m}),\psi_m}$ 
is independent of the choice of the isometry 
$(V,h) \simeq (\mathbb{F}_{q^2}^n,h_n)$, 
since the difference is a conjugation by an element of 
$\mathbf{U}_V(\bF_q)$. 
\end{rem}

For a positive integer 
$r$, let $\mu_r=\bigl\{x \in \overline{\mathbb{F}}_q \mid 
x^r=1\bigr\}$.
\begin{prop}\label{prop:trzFr}
Let $1 \leq j \leq m$. We set $(j,m)=d$. 
Let $\zeta \in \mu_{q^m+1}$. 
\begin{enumerate}
\item\label{enu:zFrj} 
We have 
\[
\tr\widetilde{\rho}_{\HU_1(q^m),\psi_m}(\zeta, \mathrm{Fr}_{q^2}^j)
=\begin{cases}
q^d & \textrm{if $\zeta \in \mu_{\frac{q^m+1}{q^d+1}}$}, \\
-1 & \textrm{otherwise}. 
\end{cases}
\]
\item\label{enu:trmd}
We have 
\[
\tr\widetilde{\rho}_{\HU_1(q^m),\psi_m}(\zeta, \mathrm{Fr}_{q^2}^j)
=\tr\widetilde{\rho}_{\HU_1(q^d),\psi_d}(\zeta^{\frac{q^m+1}{q^d+1}}). 
\]
\item\label{enu:trFrj}
We have $\tr \widetilde{\rho}_{\HU_n(q^m),\psi_m}(\mathrm{Fr}_{q^2}^j)=q^{nd}$. 
\end{enumerate}
\end{prop}
\begin{proof}
The third claim follows from the K\"{u}nneth formula 
and the first one. 
The second claim follows from 
the first one and Lemma \ref{lem:hw}. 

We show the first claim. 
We consider the curve $X_m$ defined by 
$z^q+z=x^{q^m+1}$ over $\mathbb{F}_{q^2}$. 
We have the finite \'{e}tale morphism 
\[
 X_{m,1} \to X_m;\ 
(x,z) \mapsto \left( x,\sum_{i=0}^{m-1} (-1)^i z^{q^i}\right). 
\]
The pull-back of this induces 
\begin{equation}\label{eq:mm1}
H_{\rm c}^1(X_{m,\overline{\mathbb{F}}_q},\overline{\mathbb{Q}}_{\ell})[\psi] \xrightarrow{\sim}
H_{\rm c}^1(X_{m,1,\overline{\mathbb{F}}_q},\overline{\mathbb{Q}}_{\ell})[\psi_m].   
\end{equation}
Thus 
\begin{equation}\label{m1m}
\tr\widetilde{\rho}_{\HU_1(q^m),\psi_m}(\zeta, \mathrm{Fr}_{q^2}^j)
=(-q^{-1})^j \Tr(\zeta\mathrm{Fr}_{q^2}^j;H_{\rm c}^1(X_{m,\overline{\mathbb{F}}_q},\overline{\mathbb{Q}}_{\ell})[\psi]). 
\end{equation}
In the sequel, 
we identify $\oU_1(q^i)$ with $\mu_{q^i+1}$ for any 
positive integer $i$.  
We have an isomorphism 
\begin{equation}\label{mm2}
H_{\rm c}^1(X_{m,\overline{\mathbb{F}}_q},\overline{\mathbb{Q}}_{\ell})[\psi] \simeq 
\bigoplus_{\chi \in \mu_{q^m+1}^{\vee} \setminus \{1\}}
\chi
\end{equation}
as $\mu_{q^m+1}$-representations as in \cite[Lemma 2.2]{ITlgsw1}. 
Let $\chi \in \mu_{q^m+1}^{\vee}$. We have 
\begin{equation}\label{mm3}
\chi^{q^{2j}}=\chi \iff \chi|_{\mu_{\frac{q^m+1}{q^d+1}}}
=1,  
\end{equation}
because of $(q^m+1,q^{2j}-1)=q^d+1$. 
We regard $\mu_{q^d+1}^{\vee}$ as a subset of 
$\mu_{q^m+1}^{\vee}$ by the dual of 
\[
 \mu_{q^m+1} \to \mu_{q^d+1};\ x \mapsto x^{\frac{q^m+1}{q^d+1}}. 
\] 
We have the finite morphism 
\[
 X_m \to X_d;\ (x,z) \mapsto \left( x^{\frac{q^m+1}{q^d+1}},z\right). 
\] 
Then the image of its pull-back map 
\[
H_{\rm c}^1(X_{d,\overline{\mathbb{F}}_q},\overline{\mathbb{Q}}_{\ell})[\psi] \to 
H_{\rm c}^1(X_{m,\overline{\mathbb{F}}_q},\overline{\mathbb{Q}}_{\ell})[\psi]
\]
equals $\bigoplus_{\chi \in \mu_{q^d+1}^{\vee} \setminus \{1\}} \chi$
under \eqref{mm2}. 
We write as $j=dk$. 
We compute   
\begin{align*}
\Tr(\zeta\mathrm{Fr}_{q^2}^j;H_{\rm c}^1(X_{m,\overline{\mathbb{F}}_q},\overline{\mathbb{Q}}_{\ell})[\psi])
&=\Tr(\zeta^{\frac{q^m+1}{q^d+1}}\mathrm{Fr}_{q^{2d}}^k;H_{\rm c}^1(X_{d,\overline{\mathbb{F}}_q},\overline{\mathbb{Q}}_{\ell})[\psi]) \\
&=(-q^d)^k\Tr(\zeta^{\frac{q^m+1}{q^d+1}};H_{\rm c}^1(X_{d,\overline{\mathbb{F}}_q},\overline{\mathbb{Q}}_{\ell})[\psi]) \\
&=(-q^d)^k \sum_{\chi \in \mu_{q^d+1}^{\vee} \setminus \{1\}} \chi(\zeta^{\frac{q^m+1}{q^d+1}})\\
&=\begin{cases}
(-1)^k q^{d(k+1)} & \textrm{if $\zeta \in \mu_{\frac{q^m+1}{q^d+1}}$}, \\
(-1)^{k+1} q^{dk} & \textrm{otherwise}, 
\end{cases}
\end{align*}
where the first equality follows from \eqref{mm3} and 
the second one 
follows from 
Lemma \ref{lem:FrS} \ref{enu:X11} since 
\[
 H_{\rm c}^1(X_{d,\overline{\mathbb{F}}_q},\overline{\mathbb{Q}}_{\ell})[\psi] 
 \xrightarrow{\sim}
 H_{\rm c}^1(X_{d,1,\overline{\mathbb{F}}_q},
 \overline{\mathbb{Q}}_{\ell})[\psi_d] 
\]
similarly to \eqref{eq:mm1}. 
Hence the claim follows from
\eqref{m1m}, since $d$ is odd. 
\end{proof}

\section{Compatibility}\label{od}
Let $(V,h)$ be a nondegenerate hermitian space over $\mathbb{F}_{q^2}$ 
of dimension $n$. 
\subsection{Orthogonal decomposition}
Let $(V,h)=(V_1,h_1) \oplus (V_2,h_2)$ be a decomposition of $V$
into the orthogonal direct sum of two hermitian spaces. 
Then we have the natural homomorphism 
\[
 \mathbf{HU}_{V_1} \times \mathbf{HU}_{V_2} 
 \to \mathbf{HU}_{V}. 
\]
This morphism and the projections induce 
\begin{equation}\label{4-1}
 \xymatrix{
 \bigl( \mathbf{HU}_{V_1}(\bF_{q^m}) \times 
 \mathbf{HU}_{V_2} (\bF_{q^m}) \bigr) \rtimes \Gamma 
 \ar[r]^-{\delta} \ar[d]^-{i} 
 & \mathbf{HU}_{V} (\bF_{q^m}) \rtimes \Gamma \\  
 \bigl( \mathbf{HU}_{V_1} (\bF_{q^m}) \rtimes \Gamma \bigr) \times 
 \bigl( \mathbf{HU}_{V_2} (\bF_{q^m}) \rtimes \Gamma \bigr), & 
 }
\end{equation}
where $\Gamma$ acts on $\mathbf{HU}_{V_1}(\bF_{q^m}) \times 
 \mathbf{HU}_{V_2} (\bF_{q^m})$ by $(h_1,h_2) \mapsto 
 (\sigma' h_1, \sigma' h_2)$ for $\sigma' \in \Gamma$.

\begin{prop}\label{prop:res}
The inflation of $\widetilde{\rho}_{\mathbf{HU}_V (\bF_{q^m}),\psi_m}$ by $\delta$
is isomorphic to the restriction of 
$\widetilde{\rho}_{\mathbf{HU}_{V_1} (\bF_{q^m}),\psi_m} \boxtimes \widetilde{\rho}_{\mathbf{HU}_{V_2} (\bF_{q^m}),\psi_m}$ 
by $i$. 
\end{prop}
\begin{proof}
The claim follows from Definition \ref{def:Vh} and 
the K\"{u}nneth formula. 
\end{proof}

\subsection{Parabolic subgroup}
Let $W$ be an isotropic subspace of 
$V$ and $W^{\perp}$ be its orthogonal. 
Let $(W_0,h_0)$ be the hermitian space 
$W^{\perp}/W$ with induced 
hermitian form by $h$. 
Let $\mathbf{P}_{W}$ be the stabilizer of $W$
in $\mathbf{U}_{V}$. 
Then $\mathbf{P}_{W}$ naturally acts on 
$\mathbf{H}_{W^{\perp}}$. 
We have the natural homomorphism   
$\mathbf{H}_{W^{\perp}} \rtimes \mathbf{P}_{W} 
 \to \mathbf{HU}_{W_0}$. 
This induces the homomorphism 
\begin{equation}\label{infl}
 (\mathbf{H}_{W^{\perp}} \rtimes \mathbf{P}_{W})(\bF_{q^m}) \rtimes \Gamma 
 \to \mathbf{HU}_{W_0} (\bF_{q^m}) \rtimes \Gamma. 
\end{equation}

\begin{prop}\label{prop:par}
The restriction of $\widetilde{\rho}_{\mathbf{HU}_V (\bF_{q^m}),\psi_m}$ 
to $(\mathbf{H}_{V} \rtimes \mathbf{P}_{W})(\bF_{q^m}) \rtimes \Gamma$
is isomorphic to 
\[
 \mathrm{Ind}_{(\mathbf{H}_{W^{\perp}} \rtimes \mathbf{P}_{W})(\bF_{q^m}) \rtimes \Gamma}^{(\mathbf{H}_{V} \rtimes \mathbf{P}_{W})(\bF_{q^m}) \rtimes \Gamma} \widetilde{\rho}_{\mathbf{HU}_{W_0} (\bF_{q^m}),\psi_m},
\]
where the inflation of $\widetilde{\rho}_{\mathbf{HU}_{W_0} (\bF_{q^m}),\psi_m}$ via \eqref{infl} is denoted by the same symbol. 
\end{prop}
\begin{proof}
We have $\tr\widetilde{\rho}_{\mathbf{HU}_V (\bF_{q^m}),\psi_m}(\sigma)=q^n$
by Proposition \ref{prop:trzFr} \ref{enu:trFrj}. 
Hence it suffices to check that the 
trace of the induced representation at $\sigma$ is $q^n$ by \cite[Theorem 3.3(b)]{GerWeil}. 
We see that the trace equals 
\[
\lvert \mathbf{H}_V(\bF_q)/\mathbf{H}_{W^{\perp}}(\bF_q) 
 \rvert q^{n_0}=q^n, 
\] 
where $n_0=\dim W_0$, 
by the character formula of an induced representation 
and Proposition \ref{prop:trzFr} \ref{enu:trFrj} 
(\cf \cite[the proof of Proposition 6.2]{HeWaWeil}). 
Hence we obtain the claim. 
\end{proof}

\section{Support of trace of extended Weil representation}\label{Sup}
We have the isomorphism 
\[
 \iota \colon \mathbf{HU}_{(V,h)}(\bF_{q^m}) \rtimes \Gamma \xrightarrow{\sim}
 \mathbf{HU}_{(V,-h)}(\bF_{q^m}) \rtimes \Gamma;\ 
 (v,a,g) \mapsto (v,-a,g) . 
\]
\begin{lem}\label{iso}
The $\mathbf{HU}_{(V,h)}(\bF_{q^m}) \rtimes \Gamma$-representation
$\widetilde{\rho}_{\mathbf{HU}_{(V,h)}(\bF_{q^m}),\psi_m^{-1}}$ is isomorphic to the inflation of 
$\widetilde{\rho}_{\mathbf{HU}_{(V,-h)}(\bF_{q^m}),\psi_m}$ by $\iota$. 
\end{lem}
\begin{proof}
This follows from 
Lemma \ref{lem:hw} and Proposition \ref{prop:trzFr} \ref{enu:trFrj}. 
\end{proof}

Let $(2V,\pm h)=(V,h) \oplus (V,-h)$ be the orthogonal 
direct sum. 
By \eqref{4-1}, we have the group 
homomorphism 
\[
 \delta' \colon (\mathbf{HU}_{(V,h)}(\bF_{q^m}) \times \mathbf{HU}_{(V,-h)}(\bF_{q^m})) \rtimes 
\Gamma \to \mathbf{HU}_{(2V,\pm h)}(\bF_{q^m}) \rtimes \Gamma. 
\]
Let 
\[
\Delta \colon \mathbf{HU}_{(V,h)}(\bF_{q^m})\rtimes \Gamma \to 
 \mathbf{HU}_{(2V,\pm h)}(\bF_{q^m}) \rtimes \Gamma;\ (g,\sigma) \mapsto 
 \delta'(g,\iota(g,\sigma)).
 \]
\begin{lem}\label{lem:infd}
The inflation of $\widetilde{\rho}_{\mathbf{HU}_{(2V,\pm h)}(\bF_{q^m}),\psi_m}$ by $\Delta$
is isomorphic to the restriction of 
$\widetilde{\rho}_{\mathbf{HU}_{(V,h)}(\bF_{q^m}),\psi_m} \otimes \widetilde{\rho}_{\mathbf{HU}_{(V,h)}(\bF_{q^m}),\psi_m^{-1}}$. 
\end{lem}
\begin{proof}
The inflation of 
$\widetilde{\rho}_{\mathbf{HU}_{(2V,\pm h)}(\bF_{q^m}),\psi_m}$ 
by $\delta'$ 
is isomorphic to the restriction of the representation 
\[
 \widetilde{\rho}_{\mathbf{HU}_{(V,h)}(\bF_{q^m}),\psi_m} \boxtimes 
 \widetilde{\rho}_{\mathbf{HU}_{(V,-h)}(\bF_{q^m}),\psi_m}
\] 
to 
\[
(\mathbf{HU}_{(V,h)}(\bF_{q^m}) \times 
 \mathbf{HU}_{(V,-h)}(\bF_{q^m})) \rtimes \Gamma 
\] 
by Proposition \ref{prop:res}. 
Hence the claim follows from Lemma \ref{iso}. 
\end{proof}

Let $\mathbf{Z}_{(V,h)}$ be the center of 
$\mathbf{H}_{(V,h)}$. 
We put 
$\mathbf{ZU}_{(V,h)} = \mathbf{Z}_{(V,h)} \times \mathbf{U}_{(V,h)}$. 

\begin{lem}\label{lem:indtri}
The tensor product 
$\widetilde{\rho}_{\mathbf{HU}_{(V,h)}(\bF_{q^m}),\psi_m} \otimes \widetilde{\rho}_{\mathbf{HU}_{(V,h)}(\bF_{q^m}),\psi_m^{-1}}$ is 
isomorphic to the induction of the trivial character 
of $\mathbf{ZU}_{(V,h)}(\bF_{q^m}) \rtimes \Gamma$.  
\end{lem}
\begin{proof}
Take $\zeta \in \bF_{q^2} \setminus \{ 1 \}$ such that 
$\Nr_{\bF_{q^2}/\bF_q} (\zeta)=1$. 
Let $W$ be the Lagrangian subspace 
$\{(v,\zeta v) \in V \oplus V \}$ of $2V$, 
which is isotropic via $\pm h$. 
Note that $W^{\perp}=W$. 
Applying Proposition \ref{prop:par} in this situation, 
we have $W_0=\{0\}$, 
$\mathbf{HU}_{W_0}(\bF_{q^m}) \simeq \bF_{q^m,+}$
and $\widetilde{\rho}_{\mathbf{HU}_{W_0}(\bF_{q^m}),\psi_m}$
is the character 
\[
\bF_{q^m,+}\rtimes\Gamma \to\overline{\bQ}_{\ell}^{\times};\ 
(x,\sigma^i) \mapsto \psi_m(x). 
\]
Let 
\[
 \widetilde{\psi}_m \colon (\mathbf{H}_{(W,\pm h)} \rtimes \mathbf{P}_W)(\bF_{q^m}) \rtimes \Gamma \to \overline{\mathbb{Q}}_{\ell}^{\times};\ 
 \left((v,\zeta v),a,g,\sigma^i \right) \mapsto \psi_m(a). 
\]
Then $\widetilde{\psi}_m$ is the inflation 
of $\widetilde{\rho}_{\mathbf{HU}_{W_0}(\bF_{q^m}),\psi_m}$
via \eqref{infl}. 
Thus Proposition \ref{prop:par} gives an isomorphism 
\begin{equation}\label{eq:Ind2VW}
 \widetilde{\rho}_{\mathbf{HU}_{(2V,\pm h)}(\bF_{q^m}),\psi_m}|_{(\mathbf{H}_{(2V,\pm h)} \rtimes \mathbf{P}_W )(\bF_{q^m}) \rtimes \Gamma} 
 \simeq 
 \mathrm{Ind}_{(\mathbf{H}_{(W,\pm h)} \rtimes \mathbf{P}_W)(\bF_{q^m}) \rtimes \Gamma}^{(\mathbf{H}_{(2V,\pm h)} \rtimes \mathbf{P}_W)(\bF_{q^m}) \rtimes \Gamma} \widetilde{\psi}_m.  
\end{equation}
The image of $\Delta$ is contained in 
$(\mathbf{H}_{(2V,\pm h)} \rtimes \mathbf{P}_W)(\bF_{q^m}) \rtimes \Gamma$. 
The map $\Delta$ induces a bijection 
\begin{align*}
 (\mathbf{HU}_{(V,h)}(\bF_{q^m}) \rtimes \Gamma)/&
 (\mathbf{ZU}_{(V,h)}(\bF_{q^m}) \rtimes \Gamma) \xrightarrow{\sim} \\ 
 & ((\mathbf{H}_{(2V,\pm h)} \rtimes \mathbf{P}_W)(\bF_{q^m}) \rtimes \Gamma)/((\mathbf{H}_{(W,\pm h)} \rtimes \mathbf{P}_W)(\bF_{q^m}) \rtimes \Gamma). 
\end{align*}
Inflating \eqref{eq:Ind2VW} 
under $\Delta$, 
we obtain the claim by 
Lemma \ref{lem:infd}. 
\end{proof}

\begin{rem}
The proof of Lemma \ref{lem:indtri} 
is slightly different from that of \cite[Lemma 7.3]{HeWaWeil} 
in the sense that 
we use $\{(v,\zeta v) \in V \oplus V\}$ 
instead of 
$\{(v, v) \in V \oplus V \}$ as $W$. 
By this choice, the bijection 
at the end of the proof holds and 
the proof becomes clearer. 
\end{rem}

\begin{cor}\label{cor:out}
The trace of the representation 
$\widetilde{\rho}_{\mathbf{HU}_{(V,h)}(\bF_{q^m}),\psi_m}$ 
is zero outside 
the conjugates of $\mathbf{ZU}_{(V,h)}(\bF_{q^m}) \rtimes \Gamma$. 
\end{cor}
\begin{proof}
We take an isomorphism $\overline{\mathbb{Q}}_{\ell}
\simeq \mathbb{C}$. Using this isomorphism, we consider 
$\widetilde{\rho}_{\mathbf{HU}_{(V,h)}(\bF_{q^m}),\psi_m}$ and $\widetilde{\rho}_{\mathbf{HU}_{(V,h)}(\bF_{q^m}),\psi_m^{-1}}$ as representations over $\mathbb{C}$. 
First, we show that the representation 
$\widetilde{\rho}_{\mathbf{HU}_{(V,h)}(\bF_{q^m}),\psi_m^{-1}}$
is isomorphic to the complex conjugate of 
$\widetilde{\rho}_{\mathbf{HU}_{(V,h)}(\bF_{q^m}),\psi_m}$. 
Clearly $\psi_m^{-1}$ is equal to the complex conjugate of $\psi_m$. Hence  
$\rho_{\mathbf{H}_{(V,h)}(\bF_{q^m}),\psi_m^{-1}}$ 
is isomorphic to the complex 
conjugate of $\rho_{\mathbf{H}_{(V,h)}(\bF_{q^m}),\psi_m}$. 
Therefore, 
we obtain the claim by Lemma \ref{lem:hw} and 
Proposition \ref{prop:trzFr} \ref{enu:trFrj}, 
since $\widetilde{\rho}_{\mathbf{HU}_{(V,h)}(\bF_{q^m}),\psi_m^{-1}}$ is irreducible as an $\mathbf{HU}_{(V,h)}(\bF_{q^m})$-representation. 

The required claim follows from this and Lemma \ref{lem:indtri}. 
\end{proof}

\section{Main theorem}\label{sec:Main}
Let the notation be as in \S \ref{sec:norm}. 
\begin{thm}\label{main}
Let $1 \leq i \leq m$ be an integer. 
We set $d=(m,i)$. We take an integer 
$t$ such that $ti \equiv d \pmod m$. 
Then there is a unique extension 
$\widetilde{\rho}_{\mathbf{HU}_{(V,h)}(\bF_{q^m}),\psi_m}$ of 
$\rho_{\mathbf{HU}_{(V,h)}(\bF_{q^m}),\psi_m}$ to 
$\mathbf{HU}_{(V,h)}(\bF_{q^m}) \rtimes \Gamma$ such that 
\begin{equation}\label{char}
\tr \widetilde{\rho}_{\mathbf{HU}_{(V,h)}(\bF_{q^m}),\psi_m}(g,\sigma^i)
=\tr \widetilde{\rho}_{\mathbf{HU}_{(V,h)}(\bF_{q^d}),\psi_d}(\mathrm{N}_{i,t}(g,\sigma^i))
\end{equation}
for any $g \in \mathbf{HU}_{(V,h)}(\bF_{q^m})$. 
\end{thm}

In \S \ref{GC}, we already know that 
$\widetilde{\rho}_{\mathbf{HU}_{(V,h)}(\bF_{q^m}),\psi_m}$ in Definition \ref{def:Vh} is an extension of 
$\rho_{\mathbf{HU}_{(V,h)}(\bF_{q^m}),\psi_m}$. 
Since $\rho_{\mathbf{HU}_{(V,h)}(\bF_{q^m}),\psi_m}$
is irreducible, only one extension can satisfy \eqref{char}. 
The aim in the rest of this paper is to prove that the extension 
$\widetilde{\rho}_{\mathbf{HU}_{(V,h)}(\bF_{q^m}),\psi_m}$ actually 
satisfies \eqref{char}.

\section{Reduction steps}\label{Red}
To show Theorem \ref{main}, 
we imitate the arguments in \cite[\S8]{HeWaWeil}. 
We recall a known fact. 

\begin{lem}\label{lem:ss}
Let $g_0 \in \oU (V,h)$. 
Assume that $g_0$ stabilizes no non-trivial isotropic 
subspace. Then $g_0$ is semisimple. 
\end{lem}
\begin{proof}
The claim is stated in \cite[proof of Theorem 4.9.2]{GerWeil}. 
We recall a proof here (\cf \cite[\S 8]{HeWaWeil}). 
Let $g_0 =su$ be the Jordan decomposition in $\oU (V,h)$. 
If $u \neq \id_V$, then 
$\Image (u -\id_V ) \cap \Ker (u-\id_V )$ 
is a non-trivial isotropic subspace of $V$ stable under $g_0$. 
Hence the claim follows. 
\end{proof}

We fix $i$ and $t$. 
Changing the base field from $\bF_q$ to $\bF_{q^d}$, 
we may assume that $d=(m,i)=1$ (\cf \cite[Remark 3.1 (ii)]{HeWaWeil}).
Let $g \in \mathbf{HU}_{(V,h)}(\bF_{q^m})$. 
We set 
\[
g_0=\mathrm{N}_{i,t}(g,\sigma^i). 
\]
If $g_0$ does not belong to $\mathbf{ZU}_{(V,h)}(\bF_q)$, 
the both sides of \eqref{char} are $0$ by Corollary \ref{cor:out}. 
Assume $g \in \mathbf{Z}_{(V,h)}(\bF_{q^m}) \simeq \bF_{q^m,+}$. 
The restriction of $\widetilde{\rho}_{\mathbf{HU}_{(V,h)}(\bF_{q^m}),\psi_m}$ to $\mathbf{Z}_{(V,h)}(\bF_{q^m})$ is a multiple of $\psi_m$. 
Clearly $((m-1)i/2,m)=1$ by $(m,i)=1$.
Hence the left hand side of \eqref{char} equals 
$\psi_m(g) q^n$ by Proposition \ref{prop:trzFr} \ref{enu:trFrj}
for $j=(m-1)i/2$, because 
$\sigma=\Fr_{q^2}^{\frac{m-1}{2}}$. 
By the definition of $\mathrm{N}_{i,t}$ and $(m,i)=1$,
we have 
\[
g_0=g \sigma^i(g) \cdots \sigma^{i(m-1)}(g)
=\Tr_{\mathbb{F}_{q^{2m}}/\mathbb{F}_{q^2}}(g). 
\]
Thus  
the right hand side of \eqref{char} is 
$\psi (g_0 ) \dim \widetilde{\rho}_{\HU(V,h),\psi} 
 =\psi_m(g) q^n$. 
Hence we have \eqref{char} on 
$\mathbf{Z}_{(V,h)}(\bF_{q^m}) \rtimes \Gamma$. 
By applying Lemma \ref{gy} \ref{enu:chiprod} 
to the product 
$\mathbf{ZU}_{(V,h)}=\mathbf{Z}_{(V,h)} \times \mathbf{U}_{(V,h)}$, 
we may assume that $g_0 \in \oU (V,h)$. 

We show Theorem \ref{main} in the case where 
$g_0 \in \oU (V,h)$ by induction on $n$. 
The claim for $n=1$ is shown in Proposition \ref{prop:trzFr} \ref{enu:trmd}. 
Assume $n>1$. 
If $g_0$ stabilizes a non-trivial 
orthogonal decomposition of $(V,h)$, 
the claim follows from the induction hypothesis 
by Lemma \ref{gy} \ref{enu:chiprod}, 
and Proposition \ref{prop:res}. 
If $g_0$ stabilizes a non-trivial isotropic subspace of $V$, 
the claim follows from the induction hypothesis 
by Lemma \ref{gy} \ref{enu:HInd} 
and Proposition \ref{prop:par}. 

Now, we may assume that \textit{$g_0$ stabilizes no non-trivial isotropic 
subspace and stabilizes no non-trivial orthogonal 
decomposition}. By Lemma \ref{lem:ss}, $g_0$ is semisimple. 
We write $s$ for $g_0$. 

We set $A=\mathrm{End}_{\mathbb{F}_{q^2}}(V)$. 
Let 
\[
\dagger_h \colon A \to A;\ f \mapsto f^{\dagger_h} 
\]
be the adjoint 
involution associated to $h$. Namely, we have 
$h(f(x),y)=h(x,f^{\dagger_h}(y))$ for any $x,y \in V$. 
Then we have $s^{\dagger_h}=s^{-1}$ by definition. 
Hence the involution stabilizes $\mathbb{F}_{q^2}[s]$.  
\begin{lem}\label{red}
The subalgebra $\mathbb{F}_{q^2}[s] \subset \mathrm{End}_{\mathbb{F}_{q^2}}(V)$ is a field. 
\end{lem}
\begin{proof}
Since $s$ is semisimple, we can write as 
\[
\mathbb{F}_{q^2}[s] \simeq \prod_{\alpha \in I} F_{\alpha}, 
\]
where $F_{\alpha}$ is a field. 
Let $e_{\alpha}$ be the idempotent in $\mathbb{F}_{q^2}[s]$ associated to 
$F_{\alpha}$. 
We have a direct sum 
\[
V=\bigoplus_{\alpha \in I} V_{\alpha}, 
\]
where $V_{\alpha}=\{v \in V \mid e_{\alpha} v=v\}$. 
The subspaces $V_{\alpha}$ are $s$-stable. 
The involution $\dagger_h$ gives a permutation 
$\alpha \mapsto \overline{\alpha}$ on $I$, 
with an isomorphism $F_{\alpha} \simeq F_{\overline{\alpha}}$. 

Assume that $I$ has at least two elements. 
We take $\alpha \in I$. 
If $\alpha=\overline{\alpha}$, then $V$ is the orthogonal sum of 
$V_{\alpha}$ and $\bigoplus_{\beta \neq \alpha} V_{\beta}$ as hermitian spaces. Actually, we have 
\[
h(x_{\alpha},x_{\beta})=h(e_{\alpha}x_{\alpha},x_{\beta})=h(x_{\alpha},e_{\overline{\alpha}} x_{\beta})=h(x_{\alpha},e_{\alpha} x_{\beta})=0 
\]
for any $\alpha \neq \beta$, $x_{\alpha} \in V_{\alpha}$ and $x_{\beta} \in V_{\beta}$. Hence $V$ has a non-trivial $s$-stable orthogonal decomposition. This 
contradicts to the assumption. 

Assume $\alpha \neq \overline{\alpha}$. 
Then $V_{\alpha}$ is a non-trivial $s$-stable isotropic 
subspace of $V$. 
Actually, we have 
\[
h(x_{\alpha},x'_{\alpha})=h(e_{\alpha}x_{\alpha},x'_{\alpha})=h(x_{\alpha},e_{\overline{\alpha}} x'_{\alpha})=0
\]
for any $x_{\alpha},x'_{\alpha} \in V_{\alpha}$. 
Again this contradicts to the assumption. 
Hence $I$ consists of one element. The claim follows. 
\end{proof}

We put 
\[
E=\mathbb{F}_{q^2}[s] \subset A, \quad E_+ =E^{\dagger_h}. 
\]
We regard $V$ as an $E$-vector space. 

\begin{lem}\label{inv}
\begin{enumerate}
\item 
The extension 
$E/E_+$ is a quadratic extension and 
$[E_+:\mathbb{F}_q]$ is odd. 
\item 
There exists a nondegenerate hermitian form $\widetilde{h} \colon 
V \times V \to E$ such that $h=\Tr_{E/\mathbb{F}_{q^2}} \circ \widetilde{h}$. 
\item 
We have $\dim_E V=1$. 
\end{enumerate}
\end{lem}
\begin{proof}
The first claim follows from that 
$\dagger_h$ on $\mathbb{F}_{q^2}$ is the $q$-th power map. 

Since $\Tr_{E/\bF_{q^2}} \colon E \times E \to \bF_{q^2}$ 
is nondegenerate, 
we can define a nondegenerate hermitian form 
$\widetilde{h} \colon V \times V \to E$ 
by the condition that 
\[
\Tr_{E/\bF_{q^2}}(a \widetilde{h}(v,v')) =h(v,a v') 
\]
for $a \in E, v,v' \in V$. 
Hence we obtain the second claim. 

The element $s \in E^{\times}$ acts on $V$ as 
scalar multiplication. 
Since $s$ stabilizes no non-trivial orthogonal decomposition of $(V,\widetilde{h})$ as hermitian spaces over $E$ by the assumption, 
we obtain the third claim.   
\end{proof}

\section{Proof in the reduced case}\label{Prred}
We will show Theorem \ref{main} in the situation 
of Lemma \ref{inv}. 
Note that $n$ is odd and $E=\bF_{q^{2n}}$ by Lemma \ref{inv}. 
By the natural homomorphism 
$\oU (V,\widetilde{h}) \hookrightarrow \oU (V,h)$, we obtain 
\[
 r \colon \HU(V,\widetilde{h}) \to \HU(V,h);\ 
 (v,a,g) \mapsto (v,\Tr_{E/\mathbb{F}_{q^2}}(a),g) . 
\]
We put $\psi_E =\psi \circ \Tr_{E/\mathbb{F}_{q^2}} \in 
\mathbb{F}_{q^n,+}^{\vee}$. 
\begin{lem}\label{e1}
\begin{enumerate}
\item\label{enu:inf} 
The inflation of $\rho_{\HU(V,h),\psi}$
by $r$ is isomorphic to 
$\rho_{\HU(V,\widetilde{h}),\psi_E}$.
\item\label{enu:trs} 
Let $s \in \oU (V,\widetilde{h})$. 
Then we have 
\[
\tr \rho_{\HU(V,h),\psi}(s)
=\tr \rho_{\HU(V,\widetilde{h}),\psi_E}(s)
=
\begin{cases}
q^n & \textrm{if $s=1$}, \\
-1 & \textrm{otherwise}.  
\end{cases}
\]
\end{enumerate}
\end{lem}
\begin{proof}
Let $r^\ast \rho_{\HU(V,h),\psi}$
denote the inflation of $\rho_{\HU(V,h),\psi}$
by $r$. 
The restriction of $r^\ast \rho_{\HU(V,h),\psi}$ to the center of $\oH (V,\widetilde{h})$ is a multiple of $\psi_E$ with multiplicity $q^n$. Hence we have 
\[
(r^\ast \rho_{\HU(V,h),\psi})|_{\oH (V,\widetilde{h})} \simeq \rho_{\oH (V,\widetilde{h}),\psi_E}. 
\]
Therefore the second claim implies the first one. 
The second claim follows from Lemma \ref{lem:hw}
(\cf \cite[Corollary 3.5]{GerWeil}). 
\end{proof}
Let $e=(m,n)$. 
We put $\Gamma_e=\mathrm{Gal}(\mathbb{F}_{q^e}/\mathbb{F}_q)$. 
For $\alpha \in \Gamma_e$, 
we put 
\[
E_{\alpha}=E \otimes_{\mathbb{F}_{q^e},\alpha} \mathbb{F}_{q^m}, \quad V_{\alpha}=V \otimes_{\mathbb{F}_{q^e},\alpha} \mathbb{F}_{q^m}, 
\]
where $\mathbb{F}_{q^m}$ is regarded as an $\bF_{q^e}$-algebra
via the composite $\bF_{q^e} \xrightarrow{\alpha} \bF_{q^e} \hookrightarrow 
\bF_{q^m}$. 
Note that $E_{\alpha}$ is isomorphic to $\bF_{q^{2mn/e}}$. 
We have 
\begin{equation}\label{factor}
\begin{split}
 &E \otimes_{\mathbb{F}_q} \mathbb{F}_{q^m} \xrightarrow{\sim} 
 \prod_{\alpha \in \Gamma_e}
 E_{\alpha} ;\ 
 x \otimes a \mapsto (x \otimes a)_{\alpha}, \\  
 &V \otimes_{\mathbb{F}_q} \mathbb{F}_{q^m} \xrightarrow{\sim} 
 \prod_{\alpha \in \Gamma_e}
 V_{\alpha} ;\ 
 v \otimes a \mapsto (v \otimes a)_{\alpha}. 
\end{split}
\end{equation}
The base change of $\widetilde{h}$ for 
$\mathbb{F}_{q^m}/\mathbb{F}_q$ 
induces a hermitian form 
$\widetilde{h}_{\alpha} \colon V_{\alpha} \times V_{\alpha} \to E_{\alpha}$. 
We put 
\[
h_{\alpha}=\Tr_{E_{\alpha}/\bF_{q^{2m}}} \circ \widetilde{h}_{\alpha}, 
\]
which is a hermitian form on $V_{\alpha}$ over $\mathbb{F}_{q^{2m}}$. 

Let $(V_m,h_m)$ be the base change of $(V,h)$ 
for $\bF_{q^m}/\bF_q$. 
Then $(V_m,h_m)$ is isomorphic to the 
orthogonal sum  $\bigoplus_{\alpha} (V_{\alpha},h_{\alpha})$. 
Hence we have the homomorphisms
\begin{equation}\label{ff}
 \prod_{\alpha \in \Gamma_e} 
 \HU(V_{\alpha},\widetilde{h}_{\alpha}) 
 \hookrightarrow \prod_{\alpha \in \Gamma_e} 
 \HU(V_{\alpha},h_{\alpha}) \hookrightarrow
 \HU(V_m,h_m) = 
 \mathbf{HU}_{(V,h)}(\bF_{q^m}). 
\end{equation}

We put $\alpha_0 =\id_{\bF_{q^e}} \in \Gamma_e$. 
We omit the index set $\Gamma_e$ in the notation.  
Let  
\[
\mathrm{pr}_{\alpha_0} \colon \prod_{\alpha}
 \HU(V_{\alpha},\widetilde{h}_{\alpha}) \to \HU(V_{\alpha_0},\widetilde{h}_{\alpha_0}) 
\]
be the projection. 
We consider the homomorphism 
\begin{equation}\label{eq:prid}
\left(\prod_{\alpha}\HU(V_{\alpha},\widetilde{h}_{\alpha})\right) \rtimes \langle \sigma^e\rangle
\xrightarrow{\mathrm{pr}_{\alpha_0} \times \id}
\HU(V_{\alpha_0},\widetilde{h}_{\alpha_0}) \rtimes \langle \sigma^e\rangle, 
\end{equation}
where the left hand side is regarded as a subgroup of 
$\left(\prod_{\alpha}\HU(V_{\alpha},\widetilde{h}_{\alpha})\right) \rtimes \Gamma$. 
We put 
\[
\psi_{E_{\alpha}}=\psi \circ \Tr_{E_{\alpha}/\mathbb{F}_{q^2}} \in \mathbb{F}_{q^{mn/e},+}^{\vee}. 
\]
Let $R$ denote the tensor induction to 
\[ 
\left(\prod_{\alpha}\HU(V_{\alpha},\widetilde{h}_{\alpha})\right) \rtimes \Gamma 
\] 
of the inflation of 
 $\widetilde{\rho}_{\HU(V_{\alpha_0},\widetilde{h}_{\alpha_0}),\psi_{E_{\alpha_0}}}$ under \eqref{eq:prid}. 
 Here and in the sequel, we identify 
 $\langle \sigma^e\rangle=\Gal (\bF_{q^{2m}}/\bF_{q^{2e}})$ with 
 $\Gal (\bF_{q^{2mn/e}}/\bF_{q^{2n}})$ via the natural isomorphism 
 $\bF_{q^{2m}} \otimes_{\bF_{q^{2e}}} \bF_{q^{2n}} \cong \bF_{q^{2mn/e}}$. 
 We take $s'_0 \in \oU (V_{\alpha_0},\widetilde{h}_{\alpha_0})$ with norm $s \in \oU (V,\widetilde{h})$. 
We set 
\[
s'=(s'_0,1,\ldots,1) \in \prod_{\alpha}\oU (V_{\alpha},\widetilde{h}_{\alpha}). 
\]
Recall that we assume $(m,i)=1$. 

\begin{lem}\label{pp1}
We have 
\[
\tr R(s',\sigma^i)=
\begin{cases}
q^n & \textrm{if $s=1$}, \\
-1 & \textrm{otherwise}. 
\end{cases}
\]
\end{lem}
\begin{proof}
The group $\Gamma$ permutes the factors 
$\{V_{\alpha}\}_{\alpha \in \Gamma_e}$ in \eqref{factor}
transitively and the stabilizer of each factor is $\langle \sigma^e\rangle$. 
Hence by \cite[\S2]{GlIsTen}, 
\begin{align*}
\tr R(s',\sigma^i)&=
\tr \widetilde{\rho}_{\mathrm{HU}(V_{\alpha_0},\widetilde{h}_{\alpha_0}),\psi_{E_{\alpha_0}}}
((\mathrm{pr}_{\alpha_0} \times \id)(s',\sigma^i)^e) \\ 
&=
\tr 
\widetilde{\rho}_{\mathrm{HU}(V_{\alpha_0},\widetilde{h}_{\alpha_0}),\psi_{E_{\alpha_0}}}(s'_0,\sigma^{ie})
\end{align*}
 (\cf \cite[Definitions 10 and 11]{KnoFrob}). 
 We note ${s'_0}^{\frac{q^{mn/e}+1}{q^n+1}}=s$. 
 Thus the claim follows from $(m,i)=1$ and 
 Proposition \ref{prop:trzFr} \ref{enu:zFrj} with taking 
 $(q^n,m/e)$ as $(q,m)$. 
\end{proof}
\begin{cor}\label{cor:infpro}
\begin{enumerate}
\item 
The $(\prod_{\alpha}\HU(V_{\alpha},\widetilde{h}_{\alpha}))\rtimes \Gamma$-representation $R$ is isomorphic to the inflation of 
$\widetilde{\rho}_{\mathbf{HU}_{(V,h)}(\bF_{q^m}),\psi_m}$ by the natural 
homomorphism 
\[
\left( \prod_{\alpha}\HU(V_{\alpha},\widetilde{h}_{\alpha})\right) \rtimes \Gamma \to \mathbf{HU}_{(V,h)}(\bF_{q^m}) \rtimes \Gamma. 
\]
\item\label{enu:trs'sig} 
We have 
\[
\tr \widetilde{\rho}_{\mathbf{HU}_{(V,h)}(\bF_{q^m}), \psi_m}(s',\sigma^i)
=\begin{cases}
q^n & \textrm{if $s=1$}, \\
-1 & \textrm{otherwise}. 
\end{cases}
\]
\end{enumerate}
\end{cor}
\begin{proof}
By \cite[\S 2]{GlIsTen}, the restriction of $R$
to the subgroup $\prod_{\alpha} \HU(V_{\alpha},\widetilde{h}_{\alpha})$ is isomorphic to 
$\boxtimes_{\alpha} \rho_{\HU(V_{\alpha},\widetilde{h}_{\alpha}),\psi_{E_{\alpha}}}$. 
The inflation of $\rho_{\mathbf{HU}_{(V,h)}(\bF_{q^m}),\psi_m}$
by \eqref{ff} is isomorphic to 
$\boxtimes_{\alpha} \rho_{\HU(V_{\alpha},\widetilde{h}_{\alpha}),\psi_{E_{\alpha}}}$
by Proposition \ref{prop:res} and Lemma \ref{e1} \ref{enu:inf}. 
By Proposition \ref{prop:trzFr} \ref{enu:trFrj}, we have
\[
\tr \widetilde{\rho}_{\mathbf{HU}_{(V,h)}(\bF_{q^m}),\psi_m}(\sigma)
=q^n. 
\]
Since $\sigma$ is a generator 
of $\Gamma$, 
the first claim follows from Lemma \ref{pp1}.  
The second claim follows from the first one and 
Lemma \ref{pp1}. 
\end{proof}

Note that 
$(g,\sigma^i)$ is $\mathbf{HU}_{(V,h)}(\bF_{q^m})$-conjugate to 
$(s',\sigma^i)$ 
for any $g \in \mathbf{HU}_{(V,h)}(\bF_{q^m})$ satisfying 
$s=\mathrm{N}_{i,t}(g,\sigma^i)$. 
Hence, 
it suffices to show 
\[
 \tr \widetilde{\rho}_{\mathbf{HU}_{(V,h)}(\bF_{q^m}), \psi_m}(s',\sigma^i)
 =\tr \rho_{\HU(V,h),\psi}(s) . 
\]
This follows from Lemma \ref{e1} \ref{enu:trs} and Corollary \ref{cor:infpro} \ref{enu:trs'sig}. 

\subsection*{Acknowledgements}
The authors would like to sincerely thank the reviewer for reading the paper and giving helpful comments. 
This work was supported by JSPS KAKENHI Grant Numbers 20K03529, 21H00973.


\noindent
Naoki Imai\\
Graduate School of Mathematical Sciences, The University of Tokyo, 3-8-1 Komaba, Meguro-ku, Tokyo, 153-8914, Japan \\
naoki@ms.u-tokyo.ac.jp \\[0.5cm]
Takahiro Tsushima\\ 
Keio University School of Medicine, 4-1-1 Hiyoshi, Kohoku-ku, Yokohama, 223-8521, Japan \\
tsushima@keio.jp

\end{document}